\documentclass[reqno,a4paper,11pt]{amsart}

\usepackage{amsmath,amstext,amsfonts,amsbsy,eucal,amssymb}
\usepackage[latin1]{inputenc}

\numberwithin{equation}{section}

\newtheorem{theorem}{Theorem}[section]
\newtheorem{lemma}[theorem]{Lemma}

\newtheorem{proposition}[theorem]{Proposition}
\newtheorem{corollary}[theorem]{Corollary}

\newtheorem{conjecture}[theorem]{Conjecture}

\theoremstyle{definition}

\newtheorem{remark}[theorem]{Remark}

\begin{document}

\parskip 4pt
\baselineskip 16pt


\title[Somos-4 equation and related equations]
{Somos-4 equation and related equations}

\author{Andrei K. Svinin}

\email{svinin@icc.ru}
\address{Matrosov Institute for System Dynamics and Control Theory, Siberian Branch of Russian Academy of Sciences, PO Box 292, 664033 Irkutsk, Russia}

\maketitle

\begin{abstract}
The main object of study in this paper is the well-known Somos-4 recurrence. We prove a theorem that any sequence generated by this equation  also satisfies Gale-Robinson one. The corresponding identity is written in terms of its companion elliptic sequence.  An example of such relationship  is provided by the second-order linear sequence which, as we prove using Wajda's identity, satisfies the Somos-4 recurrence with suitable coefficients. Also, we construct a class of solutions to Volterra lattice equation closely related to the second-order linear sequence.
\end{abstract}


\section{Introduction}

In this paper we  consider  three-term discrete quadratic equation of fourth order
\begin{equation}
t_nt_{n+4}=\alpha t_{n+1}t_{n+3}+\beta t_{n+2}^2,
\label{5907784}
\end{equation}
where $\alpha$ and $\beta$ are arbitrary parameters.

To start, it should be clever to recall a story about a  sequence \cite{Sloane}  http://oeis.org/A006 720, which is sometimes referred to as  Somos(4).  
It is defined by (\ref{5907784}) with $\alpha=\beta=1$ and initial data $(t_0, t_1, t_2, t_3)=(1, 1, 1, 1)$. Michael Somos once asked to prove noticed by him amazing fact of integrity of this sequence. Apparently, the proof of this fact was first published in the paper \cite{Malouf} by Malouf, but it seems that this proof does not reflect the essence of this phenomenon. 
It is now clear that this follows immediately from the Laurent property of this equation \cite{Fomin}.   
\begin{theorem}
Equation  (\ref{5907784}) has the Laurent property, that is,  given initial data $(t_0, t_1, t_2, t_3)$, all of the terms in the sequence are
Laurent polynomials in the variables $(t_0, t_1, t_2, t_3)$ whose coefficients are in $Z[\alpha, \beta]$,
so that $t_n\in \mathbb{Z}[\alpha, \beta, t_0^{\pm 1}, t_1^{\pm 1}, t_2^{\pm 1}, t_3^{\pm 1}]$ for all $n\in \mathbb{Z}$.
\end{theorem}
As is known, equation (\ref{5907784}) is not the only one that has the Laurent property. For example, it follows from the Theorem 1.6 in \cite{Fomin} that this property have any equation of the form
\begin{equation}
t_nt_{n+N}=\alpha t_{n+p}t_{n+N-p}+\beta t_{n+q}t_{n+N-q}
\label{776655}
\end{equation}
which is known as a three-term  Gale-Robinson recurrence \cite{Gale}, \cite{Gale2}, \cite{Robinson}. It should be noted that the Somos-4  and Gale-Robinson equations appear  in Fomin and Zelevinsky's theory of cluster algebras \cite{Fordy} providing corresponding combinatorial structures.

It can however be said that the Somos-4 equation (\ref{5907784}) has  its own distinctive features. Namely, each  sequence generated by  this  recurrence is associated with a sequence of points $P_0+nP$ on a some suitable  elliptic curve \cite{Hone1}. 
\begin{theorem} \label{987}
The general solution of the Somos-4 equation (\ref{5907784}) can be written in the form
\begin{equation}
t_n=AB^n\frac{\sigma(z_0+n\kappa)}{\sigma(\kappa)^{n^2}},
\label{7766557654}
\end{equation}
where $z_0,\; \kappa\in \mathrm{Jac}(E)$ are nonzero suitable complex numbers, the constants $A$ and $B$ are given by 
\[
A=\frac{t_0}{\sigma(z_0)},\; B=\frac{t_1\sigma(z_0)\sigma(\kappa)}{\sigma(z_0+\kappa)}.
\]
\end{theorem}
In Theorem \ref{987}, we denote above  by $\sigma(z)=\sigma(z; g_2, g_3)$ a corresponding  Weierstrass sigma function of an elliptic curve
$E : y^2=4x^3-g_2x-g_3$.

For completeness,  we also give well-known explicit expressions for the invariants of the curve $E$: $g_2$ and $g_3$. But first let us define a function
\[
H=\frac{t_nt_{n+3}}{t_{n+1}t_{n+2}}+\alpha\frac{t_{n+1}^2}{t_nt_{n+2}}+\alpha\frac{t_{n+2}^2}{t_{n+1}t_{n+3}}+\beta\frac{t_{n+1}t_{n+2}}{t_nt_{n+3}}.
\]
It can be checked that it is the first integral (translation invariant) of the equation  (\ref{5907784}). Apparently  this integral was first found and used in \cite{Swart}. Since we consider only autonomous equations, when writing any first integral, it is convenient to set $n$ equal to zero.

Now having in hand $H$,  one can write the invariants of the curve $E$. They look
\begin{equation}
g_2=\frac{H^4-8\beta H^2-24\alpha^2 H+16\beta^2}{12\alpha^2}
\label{763098987}
\end{equation}
and
\begin{equation}
g_3=-\frac{H^6-12\beta H^4-36\alpha^2 H^3+48\beta^2 H^2+144\alpha^2\beta H+216\alpha^4-64\beta^3}{216\alpha^3}.
\label{7630900988}
\end{equation}

In what follows we use the notion of a companion  elliptic sequence.  In general, of course, it is worth starting with the fact that an elliptic sequence, which were introduced
by Morgan Ward \cite{Ward}, is a very important concept that arises in the framework of the theory of elliptic curves. A good exposition of the relationship of elliptic sequences and elliptic curves   can be found, for example, in \cite{Swart}. To be more exactly, this sequence is defined by a quadratic recurrence 
\begin{equation}
W_nW_{n+4}=W_2^2W_{n+1}W_{n+3}-W_1W_3W_{n+2}^2,
\label{99092}
\end{equation}
which, as can be easily seen, is a special case of the Somos-4 one. From (\ref{99092}), it follows that $(W_n)$ is an anti-symmetric sequence, that is, $W_{-n}=-W_{n}$. 
In general, no one forbids to consider complex-valued sequences generated by (\ref{99092}), but from the point of view of application, integer sequences play a special role. 
In \cite{Ward}, Morgan Ward proved the following fact. 
\begin{theorem}  \label{786320}
Let $W_0=0$ and  $W_1=1$, while $W_2,\; W_3$ and $W_4$ are three arbitrary integers with the only one condition $W_2 | W_4$, then all other members of this sequence are integers and the following divisibility property holds: $W_n | W_m\;\; \mbox{whenever}\;\; n | m$. The converse is obviously true. 
\end{theorem}

Well, now goes back to the companion  elliptic sequence associated with any sequence defined by the Somos-4 equation (\ref{5907784}). It  is calculated as follows (see, for example, \cite{Hone2}). The first five of $W_n$'s  are 
\begin{equation}
W_0=0,\; W_1=1,\;  W_2=\sqrt\alpha,\; W_3=-\beta,\; W_4=-\mathcal{I}\sqrt\alpha,
\label{70000004}
\end{equation}
where $\mathcal{I}=\alpha^2+\beta H$. All other terms of this  sequence are defined by a recurrence (\ref{99092}). For reference, here are a few more terms of the sequence:
\[
W_5=-\mathcal{J},\; W_6=\beta\left(\mathcal{I}^2+\mathcal{J}\right) \sqrt\alpha ,\; W_7=\alpha^2 \mathcal{I}^3+\beta^3\mathcal{J},\; W_8=\left(\mathcal{J}^2-\beta^3\left(\mathcal{I}^2+\mathcal{J}\right)\right)\mathcal{I}\sqrt\alpha,\;
\]
\begin{equation}
W_9=-\beta\left(\alpha^2 \mathcal{I}^3\left(\mathcal{I}^2+\mathcal{J}\right)+\mathcal{J}^3\right),\ldots
\label{7764094}
\end{equation}
where $\mathcal{J}=\alpha^2\mathcal{I}-\beta^3$.

This is essentially an algebraic definition of companion elliptic sequence but we will also need its analytic definition. Attached to Theorem \ref{987}, the terms of the sequence $(W_n)$ are given as
\begin{equation}
W_n=\frac{\sigma(n\kappa)}{\sigma(\kappa)^{n^2}}.
\label{77766}
\end{equation}
where $\sigma(z)=\sigma(z; g_2, g_3)$ is the same  Weierstrass sigma function as in Theorem \ref{987}.

One of our aims of the paper is to show an infinite set of identities connecting the sequences $(t_n)$ and $(W_n)$.
In the next section, we consider Vajda's identity and its generalizations for a  Lukas sequence. 
Then we prove the validity of a three-term  four-linear identity for any Lucas sequence which makes it look very much like an elliptic sequence.  
In Section \ref{44444}, we prove Theorem \ref{850987}  which says  something like this: 1) any second-order linear sequence is a particular solution of the Somos-4 equation with suitable coefficients; 2) any second-order linear sequence satisfy   the Gale-Robinson recurrence (\ref{776655}) with coefficients that are uniquely expressed in terms of   Lucas polynomials. 
The Lucas sequence, or more precisely, its tweaked version, plays the role of the companion elliptic sequence $(W_n)$ for general linear sequence  in the sense described above.  Ultimately, this  gives us an example of a couple $(t_n, W_n)$, where $(t_n)$  is a solution of the Somos-4 equation, while $(W_n)$ is  the companion elliptic sequence.
And then we ask the following question: well,  the Gale-Robinson equation is satisfied by any second-order linear sequence, but what about the general solution of the Somos-4 equation? And as it is logical to assume, yes, such relation, that involve the companion elliptic sequence, are also true. This statement is formulated as Theorem \ref{77651}. As a corollary of this theorem, we obtain the following property of the solutions of the Somos-4 equation. Given any Somos-4 sequence $(t_n)$ select a subsequence $(t_{dn+r})$, that is,  subscripts  belong to an arbitrary arithmetic progression. Any such sequence is again a solution of the Somos-4 equation. In Section \ref{666666}, we consider   three-term Somos-$N$ equations. 
We show a theorem that provide a relationship between Somos-$N$ equation and some discrete equation that play the role of compatible constraints for a Volterra lattice equation
\begin{equation}
\frac{\partial Y_n}{\partial x}=Y_n\left(Y_{n+1}-Y_{n-1}\right).
\label{60}
\end{equation}
In Section \ref{77777}, we present a solution of Cauchy problem for the Volterra lattice (\ref{60}) with initial data  related to the second-order linear sequence.  Finally, in Sections \ref{99999} and \ref{10000000}, we discuss some technical details related to this solution.

\section{Linear sequence of the second order. Vajda's identity}

The practical purpose of this section and some subsequent sections is to show the fact that any  sequence  defined by a second-order linear recurrence 
\begin{equation}
T_{n+2}=PT_{n+1}-QT_n
\label{0921}
\end{equation}
with arbitrary initial data $(T_0, T_1)=(t_0, t_1)$ satisfies the Somos-4 equation (\ref{5907784}) as well as the Gale-Robinson recurrence (\ref{776655}).   Generally speaking, we consider any  complex-valued sequences, although of course we are aware that integer sequences generated by (\ref{0921}) are of greatest interest. Further we show the Vajda's identity for an arbitrary such sequence, and also see how generalizations of this identity lead to some identities relating second-order linear sequences to the Somos-4 recurrence.

In what follows, we  use the fact that
\begin{equation}
T_n=-t_0QD_{n-1}+t_1D_n, 
\label{0921888}
\end{equation}
where $(D_n)_{n\geq 0}$ represents a particular solution to the linear equation (\ref{0921}) with the initial condition $(t_0, t_1)=(0, 1)$. The first few members of this sequence look like this:
\[
D_0=0,\; D_1=1,\; D_2=P,\; D_3=P^2-Q,\; D_4=P\left(P^2-2Q\right),\ldots 
\]
In what follows, we call the sequence $(D_n)_{n\geq 0}$ the Lucas sequence. Of course, this sequence of polynomials in $(P, Q)$ is the source of many well-known number sequences, such as Fibonacci numbers, Mersenne numbers, etc., but we consider this sequence in a slightly different way. 

Further, it will be extremely important for us that   the Lukas sequence $(D_n)_{n\geq 0}$  can be formally  extended   backwards   by 
\begin{equation}
D_{-n}=-\frac{D_n}{Q^n}
\label{5544}
\end{equation}
and we will use this fact repeatedly. Moreover, by (\ref{0921888}) and (\ref{5544}), we have 
\[
T_{-n}=\frac{t_0D_{n+1}-t_1D_n}{Q^n}
\]
that means that any  second-order linear sequence can also be extended backwards.
Throughout the rest of this paper  we consider   only bi-infinite sequences $(X_n)_{n\in\mathbb{Z}}$, where $X$  is $T$ or $D$ or something others.  For simplicity, to denote bi-infinite sequence,  we agree to write simply $(X_n)$.

Proof of the following lemmas are simple and standard and are given here mainly for completeness.
\begin{lemma}
For the Lukas sequence $(D_n)$, the convolution identity 
\begin{equation}
D_{n+p}=-QD_{p-1}D_{n}+D_{p}D_{n+1},
\label{60967326}
\end{equation}
for any pair $(p, n)\in\mathbb{Z}^2$, is valid.
\end{lemma}

\begin{proof}  
Let us suppose that we have proved (\ref{60967326}) for some values $n-1$ and $n$. Then
\begin{eqnarray}
D_{n+1+p}&=&P\left(-QD_{p-1}D_{n}+D_{p}D_{n+1}\right)-Q\left(-QD_{p-1}D_{n-1}+D_{p}D_{n}\right) \nonumber\\
&=&QD_{p-1}\left(-PD_{n}+QD_{n-1}\right)+D_{p}\left(PD_{n+1}-QD_{n}\right) \nonumber\\
&=&-QD_{p-1}D_{n+1}+D_{p}D_{n+2}. \label{6655}
\end{eqnarray}
For $n=1$, relation (\ref{60967326}) becomes 
\begin{equation}
D_{p+1}=-QD_{p-1}+PD_{p}
\label{0000}
\end{equation}
what is valid by definition. In turn, for $n=2$, (\ref{60967326}) becomes 
\begin{eqnarray}
D_{p+2}&=&PD_{p+1}-QD_{p} \nonumber\\
&=&P\left(PD_{p}-QD_{p-1}\right)-QD_{p} \nonumber\\
&=&-QPD_{p-1}+\left(P^2-Q\right)D_{p} \nonumber \\
&=&-QD_{2}D_{p-1}+D_{3}D_{p}. \label{9876500}
\end{eqnarray}
Relations  (\ref{0000}) and (\ref{9876500}) give a basis for the induction, with the help of (\ref{6655}), to prove (\ref{60967326}) for all $n\geq 1$. To prove (\ref{60967326}) for all $n\leq 0$ one can use similar arguments. Therefore  this lemma is proved.  
\end{proof}

In the sequel, we will use the fact that the identity (\ref{60967326}) can be rewritten as
\begin{equation}
Q^{p-1}D_{n-p+1}=D_{p}D_{n}-D_{p-1}D_{n+1}.
\label{673095426}
\end{equation}

The following lemma contains an identity for the Lukas sequence $(D_n)$ that is a generalization of Wajda's identity \cite{Vajda} for the Fibonacci numbers: 
\[
F_{n+p}F_{n+q}-F_{n}F_{n+p+q}=(-1)^nF_{p}F_{q}.
\]  
\begin{lemma} \label{9809}
For the Lukas sequence $(D_n)$, a Vajda's identity 
\begin{equation}
D_{n+p}D_{n+q}-D_{n}D_{n+p+q}=Q^nD_{p}D_{q},
\label{6426}
\end{equation}
for any $(p, q, n)\in\mathbb{Z}^3$, is valid.
\end{lemma}

\begin{proof}  
By the convolution identity of the form (\ref{60967326}) and (\ref{673095426}) , we have
\begin{eqnarray*}
D_{n+p}D_{n+q}-D_{n}D_{n+p+q}&=&\left(-QD_{p-1}D_{n}+D_{p}D_{n+1}\right)D_{n+q} \\
&&-D_{n}\left(-QD_{p-1}D_{n+q}+D_{p}D_{n+q+1}\right) \\
&=&D_{p}\left(D_{n+1}D_{n+q}-D_nD_{n+q+1}\right) \\
&=&Q^nD_{p}D_q.
\end{eqnarray*}
\end{proof}  
The following lemma is a consequence of the Vajda's identity   (\ref{6426}).
\begin{lemma}
By (\ref{6426}), for the second-order linear sequence $(T_n)$, identity 
\begin{equation}
T_{n+p}T_{n+q}-T_{n}T_{n+p+q}=cQ^{n}D_pD_q,
\label{6400026}
\end{equation}
for any $(p, q, n)\in\mathbb{Z}^3$, is valid, where $c=Qt_0^2-Pt_0t_1+t_1^2$.
\end{lemma}

\begin{proof}  
This is proved by direct calculation. A little difficulty can only be to prove the identity 
\[
\left(D_{n+p-1}D_{n+q}+D_{n+p}D_{n+q-1}\right)-\left(D_{n-1}D_{n+p+q}+D_{n}D_{n+p+q-1}\right)=PQ^{n-1}D_{p}D_{q}.
\]
But, having (\ref{6426}) in hand, we can rewrite the left-hand side of this relation as
$
Q^nD_{p-1}D_q+Q^{n-1}D_{p+1}D_q
$
and after that the identity becomes obvious by virtue of the definition of the Lukas sequence.
\end{proof}

\section{Four-linear identity associated  to the Vajda's identity}


Let $n=a-c,\; p=b+c,\; q=c-b$. With (\ref{5544}) we can reformulate  Lemma \ref{9809} as
\begin{lemma} \label{65565}
For any triple of numbers $(a, b, c)\in\mathbb{Z}^3$ the Lukas sequence $(D_n)$ satisfies the identity
\begin{equation}
\frac{D_{a-b}D_{a+b}}{Q^a}+\frac{D_{b-c}D_{b+c}}{Q^b}+\frac{D_{c-a}D_{c+a}}{Q^c}=0.
\label{6428756}
\end{equation}
\end{lemma}

Let us underline  that  relation (\ref{6428756}) represents the Wajda's identity (\ref{6426}) but only  in symmetrical form. Looking at the symmetrical form of Wajda's identity (\ref{6428756}), only the blind will not notice how it can be generalized. So, let us  formulate a corresponding proposition that deviates somewhat from the main content and can rather be considered as a remark.
\begin{proposition} 
For any   $(a_1,\ldots, a_d)\in\mathbb{Z}^d$, where $d\geq 2$, the Lukas sequence $(D_n)$ satisfies the identity
\begin{equation}
\sum_{j=1}^d\frac{D_{a_j-a_{j+1}}D_{a_j+a_{j+1}}}{Q^{a_j}}=0,\;\; a_{d+1}=a_1.
\label{3421}
\end{equation}
\end{proposition} 
\begin{proof}
We prove this proposition by induction. Let us denote, for convenience, the symbol $\{a, b\}=D_{a-b}D_{a+b}/Q^a$. Note that, due to (\ref{5544}), the skew-symmetry relation $\{a, b\}=-\{b, a\}$ holds. Let us suppose we have proved the identity (\ref{3421}) for some value of $d$. Let us write the relation
\[
\sum_{j=1}^{d+1}\{a_j, a_{j+1}\}=0,\;\; a_{d+2}=a_1
\]
and then add and subtract $\{a_d, a_{1}\}$ to it. Taking into account our assumption, the last relation reduces to the relation
\[
\{a_d, a_{d+1}\}-\{a_d, a_{1}\}+\{a_{d+1}, a_{1}\}=0
\] 
and it, in turn, is an identity due to the skew-symmetry of $\{a, b\}$ and (\ref{6428756}). The induction base is given  by (\ref{6428756}). 
\end{proof}


Next we would like to show necessary for us   three-term four-linear identity associated in a sense to the Vajda's identity (\ref{6428756}). 
\begin{lemma} \label{7765}
By (\ref{6428756}), for any $(a_1, a_2, a_3, b)\in\mathbb{Z}^4$, the sequence $(D_n)$ satisfies the identity
\begin{equation}
\frac{D_{b-a_3}D_{b+a_3}D_{a_1-a_2}D_{a_1+a_2}}{Q^{a_1}}+\frac{D_{b-a_1}D_{b+a_1}D_{a_2-a_3}D_{a_2+a_3}}{Q^{a_2}} +\frac{D_{b-a_2}D_{b+a_2}D_{a_3-a_1}D_{a_3+a_1}}{Q^{a_3}}=0.
\label{34211111}
\end{equation}
\end{lemma}

\begin{proof}
Using symbol $\{a, b\}$, we rewrite relation (\ref{34211111}) that need to be proven as
\begin{equation}
\{b, a_3\}\{a_1, a_2\}+\{b, a_1\}\{a_2, a_3\}+\{b, a_2\}\{a_3, a_1\}=0.
\label{3420011}
\end{equation}
We already have, for any $(a, b, c)\in\mathbb{Z}^3$,  two identities
\begin{equation}
\{a, b\}+\{b, a\}=0\;\;\mbox{and}\;\; \{a, b\}+\{b, c\}+\{c, a\}=0.
\label{342761}
\end{equation}
Subtracting from  (\ref{3420011}) the relation
\[
\{b, a_3\}\left(\{a_1, a_2\}+\{a_2, a_3\}+\{a_3, a_1\}\right)=0,
\]
we get the following one:
\[
\{a_2, a_3\}\left(\{b, a_1\}-\{b, a_3\}\right)+\{a_3, a_1\}\left(\{b, a_2\}-\{b, a_3\}\right)=0.
\]
which,  we can rewrite, by virtue of (\ref{342761}), as
\[
\{a_2, a_3\}\{a_3, a_1\}+\{a_3, a_1\}\{a_3, a_2\}=0
\]
that is obviously an identity. 
\end{proof}
\begin{remark}
One sees that the proved identity (\ref{34211111}) is, in a certain sense, attached to the identity (\ref{6428756}), which in turn  is a special case of identity (\ref{3421}) for $d=3$. There are in fact  an infinite number of identities of the form
\begin{equation}
\sum_{j=1}^dD_{b-a_{\lambda_j}}D_{b+a_{\lambda_j}}\frac{D_{a_j-a_{j+1}}D_{a_j+a_{j+1}}}{Q^{a_j}}=0,\;\;a_{d+1}=a_1,
\label{31}
\end{equation}
where $\lambda_j\in(1,\ldots, d)$, attached to identities (\ref{3421}) with $d\geq 3$. The proof of each  is similar to the proof of  Proposition \ref{7765}. As we already found out, we have solution $(\lambda_1, \lambda_2, \lambda_3)=(3, 1, 2)$.  
The following problem appears here: some suitable sets can be  equivalent. Indeed, if $(\lambda_1,\ldots, \lambda_d)$ is a suitable set, then so is $(\lambda_d+1, \lambda_1+1,\ldots, \lambda_{d-1}+1)$. One need to keep in mind that if appears $d+1$, it must be replaced by 1. In the case of $d=3$, the set $(\lambda_1, \lambda_2, \lambda_3)=(3, 1, 2)$ by  a permutation with a shift is translated into itself. In the case $d=4$, we have four equivalent sets, namely, $(4, 4, 1, 3)\rightarrow (4, 1, 1, 2)\rightarrow (3, 1, 2, 2)\rightarrow (3, 4, 2, 3)$, and therefore in the case $d=4$,  we have only one identity of the form (\ref{31}) with $(\lambda_1, \lambda_2, \lambda_3, \lambda_4)=(4, 4, 1, 3)$. One can pose the task of listing identities of the form (\ref{31}), but the discussion of this problem will take us far from the main content of the paper.
\end{remark}

\section{Linear second-order sequence as a particular solution of the Somos-4 equation}
\label{44444}


Let $n=a_1-b,\; u=a_2-b,\; m=a_3-b,\; s=2b$. With these new variables,  the proved identity (\ref{34211111}) can be rewritten in the following form:
\begin{equation}
Q^uD_{m}D_{m+s}D_{n-u}D_{n+u+s}+Q^mD_{n}D_{n+s}D_{u-m}D_{u+m+s} +Q^nD_{u}D_{u+s}D_{m-n}D_{m+n+s}=0. 
\label{870009}
\end{equation}
\begin{remark}
If we put $s=0$ and $r=1$, then (\ref{870009})  reduced to the form
\begin{equation}
Q^{q-1}D_{n-u}D_{n+u}=D_{q}^2D_{n-1}D_{n+1}-D_{u-1}D_{u+1}D_{n}^2. 
\label{809}
\end{equation}
This identity has been presented and used in \cite{Lukas} by \'Eduard Lukas (see also \cite{Bell}). 
\end{remark}

With (\ref{5544}) let us slightly correct the identity (\ref{870009}) to have it in the following form:
\begin{equation}
Q^mD_{u}D_{u+s}D_{n-m}D_{n+m+s}-Q^uD_{m}D_{m+s}D_{n-u}D_{n+u+s} +Q^uD_{m-u}D_{u+m+s}D_{n}D_{n+s}=0. 
\label{09}
\end{equation}

\begin{lemma}
By (\ref{09}), identity
\begin{equation}
Q^mD_{u}D_{u+s}T_{n-m}T_{n+m+s}-Q^uD_{m}D_{m+s}T_{n-u}T_{n+u+s} +Q^uD_{m-u}D_{u+m+s}T_{n}T_{n+s}=0. 
\label{59084}
\end{equation}
holds.
\end{lemma}

\begin{proof}
This lemma can be  proven by direct calculation. Let us substitute  (\ref{0921888})  into (\ref{59084}).  It is obvious that collecting terms at  $t_1^2$ we get the proven identity (\ref{09}), while
at $t_0^2$, we get again this  identity but with shifted $n$. It remains  to see what we have at $t_0t_1$. Only at this stage there are some difficulties. Namely we have the following relation:
\begin{eqnarray}
&& Q^mD_{u}D_{u+s}\left(D_{n-m}D_{n+m+s-1}+D_{n-m-1}D_{n+m+s}\right) \nonumber\\[0.2cm]
&&\;\;\;\;\;\;\;\;\;\;\;\;\;\;\;\;\;\;\;-Q^uD_{m}D_{m+s}\left(D_{n-u}D_{n+u+s-1}+D_{n-u-1}D_{n+u+s}\right) \nonumber\\[0.2cm]
&&\;\;\;\;\;\;\;\;\;\;\;\;\;\;\;\;\;\;\;\;\;\;\;\;\;\;\;\;\;\;\;\;\; +Q^uD_{m-u}D_{u+m+s}\left(D_{n}D_{n+s-1}+D_{n-1}D_{n+s}\right)=0 
\label{7765498}
\end{eqnarray}
that needs a proof. Let us first prove that
\begin{eqnarray}
&& Q^mD_{u}D_{u+s}D_{n-m}D_{n+m+s-1} -Q^uD_{m}D_{m+s}D_{n-u}D_{n+u+s-1}+Q^uD_{m-u}D_{u+m+s}D_{n}D_{n+s-1}\nonumber \\ [0.2cm]
&&\;\;\;\;\;\;\;\;\;\;\;\;\;\;\;\;\;\;\;\;\;\;\;\;\;\;\;\;\;\;\;\;\;\;\;\;\;=-Q^{n+u+s-1}D_uD_mD_{m-u}.
\label{094532}
\end{eqnarray}
By Vajda's identity (\ref{6426}), we can rewrite the left-hand side of relation  (\ref{094532}) as 
\begin{eqnarray*}
LHS(\ref{094532})&=&\left(Q^mD_{u}D_{u+s} -Q^uD_{m}D_{m+s}+Q^uD_{m-u}D_{u+m+s}\right)D_{n}D_{n+s-1} \\[0.2cm]
&&+Q^{m+n}D_{u}D_{u+s}D_{-m}D_{m+s-1}-Q^{u+n}D_{m}D_{m+s}D_{-u}D_{u+s-1} \\[0.2cm]
&=&Q^{m+n}D_{u}D_{u+s}D_{-m}D_{m+s-1}-Q^{u+n}D_{m}D_{m+s}D_{-u}D_{u+s-1} \\[0.2cm]
&=&Q^nD_uD_m\left(D_{m+s}D_{u+s-1}-D_{m+s-1}D_{u+s}\right)
\end{eqnarray*}
Here we take into account that, due to Wajda's identity (\ref{6426}), the sum in parentheses at $D_{n}D_{n+s-1}$ is equal to zero.  With the identity
\[
D_{m+s}D_{u+s-1}-D_{m+s-1}D_{u+s}=Q^{m+s-1}D_{u-m},
\]
that follows from (\ref{673095426}), we get  (\ref{094532}). 

It remains to prove the identity
\begin{eqnarray}
&& Q^mD_{u}D_{u+s}D_{n-m-1}D_{n+m+s} -Q^uD_{m}D_{m+s}D_{n-u-1}D_{n+u+s}
+Q^uD_{m-u}D_{u+m+s}D_{n-1}D_{n+s}
\nonumber \\ [0.3cm]
&&\;\;\;\;\;\;\;\;\;\;\;\;\;\;\;\;\;\;\;\;\;\;\;\;\;\;\;\;\;\;\;\;\;\;\;\;\; =Q^{n+u+s-1}D_uD_mD_{m-u}, 
\label{8876590}
\end{eqnarray}
but this is done by analogy with the proof of the identity (\ref{094532}). And now summing the left and right sides of the identities (\ref{094532}) and (\ref{8876590}), we get  (\ref{7765498}).
\end{proof}


Putting  $m=2$, $u=1$ and $s=0$, in the  identity (\ref{59084}), we obtain the fact that the second-order linear sequence $(T_n)$ satisfies an identity
\begin{equation}
T_{n}T_{n+4}=\frac{P^2}{Q}T_{n+1}T_{n+3}-\frac{P^2-Q}{Q}T_{n+2}^2
\label{590800084}
\end{equation}
which, in turn, means that this  sequence  is a particular solution of the  Somos-4 recurrence (\ref{5907784}) with coefficients
\begin{equation}
\alpha=\frac{P^2}{Q},\; \beta=-\frac{P^2-Q}{Q}
\label{65420}
\end{equation} 
and initial data 
\begin{equation}
(t_0, t_1, t_2, t_3)=\left(t_0, t_1, -t_0Q+t_1P, -t_0QP+t_1\left(P^2-Q\right)\right).
\label{65765090}
\end{equation} 
Now we would like to see what companion elliptic sequence corresponding to the sequence $(T_n)$ is. Direct calculation gives
\[
H=\frac{P^2+2Q}{Q}\;\; \mbox{and}\;\; \mathcal{I}=\alpha^2+\beta H=-\frac{P^2-2Q}{Q}.
\]
Then, by (\ref{70000004}), we have
\[
W_2=\sqrt \alpha=\frac{P}{\sqrt Q},\;\; W_3=-\beta=\frac{P^2-Q}{Q},\; W_4=-\mathcal{I}\sqrt \alpha=\frac{P\left(P^2-2Q\right)}{\sqrt[3]{Q}}. 
\]
One sees that for  $n$, from zero to four, 
\begin{equation}
W_n=\frac{D_n}{Q^{(n-1)/2}}.
\label{763098}
\end{equation}
Inspired by this we substitute $D_n=Q^{(n-1)/2}W_n$ into (\ref{870009}), to get the following relation:
\begin{equation}
W_{m}W_{m+s}W_{n-u}W_{n+u+s}+W_{n}W_{n+s}W_{u-m}W_{u+m+s} +W_{u}W_{u+s}W_{m-n}W_{m+n+s}=0. 
\label{843009}
\end{equation}
But it is known that this relation is satisfied by the elliptic sequence \cite{Stange}. In particular, putting $s=0$ and $m=1$, we get well-known equation
\[
W_{n-u}W_{n+u}=W_{u}^2W_{n-1}W_{n+1}-W_{u-1}W_{u+1}W_{n}^2
\]
for elliptic sequences \cite{Ward}.

Let us now summarize all of the above in the following theorem.
\begin{theorem}  \label{850987}
Linear sequence of the second order $(T_n)$ is a particular solution of the Somos-4 recurrence (\ref{5907784}) with the coefficients  $\alpha$ and $\beta$ given by (\ref{65420}) and the initial data given by (\ref{65765090}). This sequence satisfies also
\begin{equation}
W_{u}W_{u+s}t_{n-m}t_{n+m+s}-W_{m}W_{m+s}t_{n-u}t_{n+u+s} +W_{m-u}W_{u+m+s}t_{n}t_{n+s}=0,
\label{599964084}
\end{equation}
where $W_n$ is given by (\ref{763098}).
\end{theorem}
\begin{remark}
It should be noted that formula (\ref{763098}) can be found on the first page of the Morgan Ward's paper \cite{Ward}.  Together with others, it gives there an example of degenerate  elliptic sequence.
\end{remark}

With the help of (\ref{763098987}) and (\ref{7630900988}), we calculate the invariants of the elliptic curve. In this case, they look like this:
\[
g_2=\frac{(P^2-4Q)^2}{12Q^2}\;\;\mbox{and}\;\; g_3=-\frac{(P^2-4Q)^3}{216 Q^3}.
\]
Moreover, in this case, the degeneracy relation $g_2^3-27g_3^2=0$ is satisfied. In turn, the equation of a degenerate elliptic curve is as follows:
\[
y^2=4\left(x-\frac{P^2-4Q}{12Q}\right)^2\left(x+\frac{P^2-4Q}{6Q}\right).
\]

\section{Theorem that every Somos-4 sequence is a Gale-Robinson one}


So we have proved  Theorem \ref{850987} that any second-order linear sequence is a solution of the Somos-4 recurrence and moreover, it satisfies relation (\ref{599964084}) in which its companion degenerate elliptic sequence is involved. Now it is natural to assume that all this takes place in the non-degenerate case. The following theorem   says that this is so.
\begin{theorem} \label{77651} 
The general solution of the Somos-4 equation (\ref{5907784}) satisfy
\begin{equation}
W_{q-p}W_{N-p-q}t_{n}t_{n+N}=W_qW_{N-q}t_{n+p}t_{n+N-p}-W_{p}W_{N-p}t_{n+q}t_{n+N-q}.
\label{9909984322}
\end{equation}
for any $(N, p, q, n)\in \mathbb{Z}^4$, where $(W_n)$ is the companion  elliptic sequence defined above.
\end{theorem} 
\begin{proof}
One can say that the idea of the proof of Theorem \ref{77651} presented below follows the lines proposed in \cite{Hone4}.
Given any quadruple of numbers $(s, m, u, n)\in \mathbb{Z}^4$, we define 
\[
f=\frac{\sigma(\kappa)^{2g}}{A^2B^{2n+s}},
\]
where $g=n^2+m^2+u^2+s^2+s\left(n+m+u\right)$ and then substitute  (\ref{776655}) and  Ward's expression  (\ref{77766}) 
for the terms of the companion  elliptic sequence into the relation 
\[
f\cdot\left(W_{u}W_{u+s}t_{n-m}t_{n+m+s}-W_{m}W_{m+s}t_{n-u}t_{n+u+s} +W_{m-u}W_{u+m+s}t_{n}t_{n+s}\right)=0.
\]
As a result of rather tedious transformations we get not such a terrible expression, namely,
\[
\sigma(u\kappa)\sigma((s+u)\kappa)\sigma(z_0+(n-m)\kappa)\sigma(z_0+(n+s+m)\kappa) 
\]
\[
\;\;\;\;\;\;\;\;-\sigma(m\kappa)\sigma((s+m)\kappa)\sigma(z_0+(n-u)\kappa)\sigma(z_0+(n+s+u)\kappa) \nonumber \\
\]
\begin{equation}
\;\;\;\;\;\;\;\;\;\;\;\;\;\;\;\;\;\;\;\;\;\;\;\;+\sigma((m-u)\kappa)\sigma((m+s+u)\kappa)\sigma(z_0+n\kappa)\sigma(z_0+(n+s)\kappa)=0.
\label{5543}
\end{equation}
Let us now define
\[
a=z_0+\left(n+\frac{s}{2}\right)\kappa,\; b=\left(m+\frac{s}{2}\right)\kappa,\; c=\left(u+\frac{s}{2}\right)\kappa,\; d=-\frac{s}{2}\kappa.
\]
This gives us the opportunity to rewrite (\ref{5543}) as
\[
\sigma(c+d)\sigma(c-d)\sigma(a+b)\sigma(a-b)
-\sigma(b+d)\sigma(b-d)\sigma(a+c)\sigma(a-c)
\]
\[
+\sigma(b+c)\sigma(b-c)\sigma(a+d)\sigma(a-d)=0.
\]
But this last relation is the well-known identity for the Weierstrass sigma function. To prove Theorem \ref{77651}, it remains to make the substitution $u=q-p$, $s=N-2q$, $m=q$ and shift $n\rightarrow n+q$.
\end{proof}

It should be noted that, in two special cases $N=2p$ and $N=2p+1$,  with  $q=p+1$ in both cases, from (\ref{9909984322}),   we immediately obtain the following two formulas which are contained in  Theorem 3 of the paper \cite{Poorten}:
\[
W_1^2t_{n-p}t_{n+p}=W_{p}^2t_{n-1}t_{n+1}-W_{p-1}W_{p+1}t_{n}^2
\]
and
\[
W_1W_2t_{n-p}t_{n+p+1}=W_{p}W_{p+1}t_{n-1}t_{n+2}-W_{p-1}W_{p+2}t_{n}t_{n+1}.
\]

Let us now show a consequence of this theorem.
\begin{corollary}
Given any solution of Somos-4 recurrence, let us define a sequence $(t_{d, n})$ by  $t_{d, n}=t_{dn+r}$, where $d$ is an arbitrary nonzero positive integer  and $r=0,\ldots, d-1$. Then the sequence $\left(t_{d, n}\right)$ satisfies Somos-4 recurrence
\begin{equation}
t_{d, n}t_{d, n+4}=\alpha_d t_{d, n+1}t_{d, n+3}+\beta_d t_{d, n+2}^2.
\label{7653096}
\end{equation}
\end{corollary}
\begin{proof}
Let $N=4d$, $p=d$ and $q=2d$. Substituting these value into (\ref{9909984322}), we get 
\[
W_{d}^2t_{n}t_{n+4d}=W_{2d}^2t_{n+d}t_{n+3d}-W_{d}W_{3d}t_{n+2d}^2=0
\]
and now replacing $n\rightarrow dn+r$ we get (\ref{7653096}), where
\begin{equation}
\alpha_d=\frac{W_{2d}^2}{W_{d}^2}\;\;\mbox{and}\;\; \beta_d=\frac{W_{3d}}{W_{d}}.
\label{7688896}
\end{equation}
\end{proof}

For example, calculating, with the help of (\ref{70000004}), (\ref{7764094}) and  (\ref{7688896}),   yields 
\[
\alpha_2=\mathcal{I}^2,\;\; \beta_2=\beta\left(\mathcal{I}^2+\mathcal{J}\right)
\]
and  
\[
\alpha_3=\alpha\left(\mathcal{I}^2+\mathcal{J}\right)^2,\;\;
\beta_3=\alpha^2 \mathcal{I}^3\left(\mathcal{I}^2+\mathcal{J}\right)+\mathcal{J}^3,
\]
where, by definition,  $\mathcal{J}=\alpha^2 \mathcal{I}-\beta^3$.

\section{Three-term Somos-$N$ equation related  equation}

\label{666666}

In a particular case $p=1$ and $q=2$, the Gale-Robinson recurrence (\ref{776655}) becomes
\begin{equation}
t_{n}t_{n+N}=\alpha_{N} t_{n+1}t_{n+N-1}+\beta_{N} t_{n+2}t_{n+N-2},\; N\geq 4.
\label{5084}
\end{equation}
We are forced to number the constants $\alpha$ and $\beta$ in (\ref{5084}), as it is convenient for the formulation of further statements. 
We will call relation (\ref{5084}) a three-term  Somos-$N$ equation. Further we will show that any solution of the Somos-$N$ equation  gives a solution to some other  equation, which in turn represents a relation consistent with a Volterra lattice  equation.

The following lemma can be proven by direct calculation.
\begin{lemma}
One of the first integral for the Somos-$N$ equation (\ref{5084}) is 
\begin{equation}
H_{N}=\sum_{j=0}^{N-4}\frac{t_{j}t_{j+3}}{t_{j+1}t_{j+2}}+\alpha_{N}\frac{t_{1}t_{N-3}}{t_0t_{N-2}}+\alpha_{N}\frac{t_{2}t_{N-2}}{t_{1}t_{N-1}}+\beta_{N}\frac{t_{2}t_{N-3}}{t_0t_{N-1}}.
\label{654009}
\end{equation}
\end{lemma}

On the other hand, given any integer $N\geq 4$, let us consider an   equation 
\begin{equation}
y_{n+1}\left(\sum_{j=0}^{N-3}y_{n+j}-H_{N}\right)=y_{n+N-2}\left(\sum_{j=0}^{N-3}y_{n+j+2}-H_{N}\right),
\label{675430}
\end{equation}
where $H_{N}$ are supposed to be some constant. 
\begin{remark}
It may be puzzling to the fact that now and further we use the same notations that  we have already used above for   another objects, like $H_{N}$, $\beta_{N}$ etc. But one have to be a little patient because  it will later turn out that these notations  are used in fact for the same things.
\end{remark}

We would like to spend some lines to show how equation (\ref{675430}) is related to the Somos-$N$ equation (\ref{5084}).
From the very form of the equation, it is clear that it has the following first integral:
\begin{equation}
\beta_{N}=\prod_{j=1}^{N-3} y_{n+j}\left(\sum_{j=0}^{N-2}y_{n+j}-H_{N}\right).
\label{0986504}
\end{equation}
Moreover we  write down one more not obvious the first integral \cite{Svinin2}, \cite{Hone3}.
\begin{lemma}
A quantity 
\[
\mathcal{I}_{N}=\prod_{j=0}^{N-2} y_{n+j}+\sum_{j=1}^{N-3} y_{n+j}\left(\sum_{j=0}^{N-2} y_{n+j}-H_{N}\right)\prod_{j=1}^{N-3} y_{n+j}.
\]
is the first integral for equation (\ref{675430}).
\end{lemma}

We have every right to consider relation (\ref{0986504}) as an equation with two constants $H_N$ and $\beta_N$. With a substitution $y_n=f_nf_{n+1}$, it becomes
\begin{equation}
\sum_{j=0}^{N-2} f_{n+j}f_{n+j+1}=\frac{\beta_{N}}{\prod_{j=1}^{N-3} f_{n+j}f_{n+j+1}}+H_{N}.
\label{0665404}
\end{equation}
\begin{lemma}
A function 
\begin{equation}
\alpha_{N, n}=\prod_{j=0}^{N-2}f_{n+j}-\frac{\beta_N}{\prod_{j=1}^{N-3}f_{n+j}}.
\label{08804}
\end{equation}
is a 2-integral for equation (\ref{0665404}). Moreover, we have
\begin{equation}
\alpha_{N, n+1}=\prod_{j=1}^{N-3}f_{n+j}\left(H_{N}-\sum_{j=0}^{N-3} f_{n+j}f_{n+j+1}\right).
\label{07434}
\end{equation}
\end{lemma}
Of course, the product $\alpha_n^{(k)}\alpha_{n+1}^{(k)}$ is the first integral. 
\begin{lemma}
We have a relation $\alpha_{N, 0}\alpha_{N, 1}=\mathcal{I}_{N}-H_{N}\beta_{N}$.
\end{lemma}
Let us now rewrite (\ref{08804}) as
\[
\alpha_{N, n}\prod_{j=1}^{N-3} f_{n+j}=\prod_{j=0}^{N-3} f_{n+j}f_{n+j+1}-\beta_N.
\]
Finally, substituting, $f_n=t_nt_{n+2}/t_{n+1}^2$, into the latter, we get relation
\begin{equation}
t_nt_{n+N}=\alpha_{N, n}t_{n+1}t_{n+N-1}+\beta_N t_{n+2}t_{n+N-2}.
\label{0004}
\end{equation}
One sees that the last relation is nothing more than the Somos-$N$ equation (\ref{5084}), only with a 2-periodic coefficient $\alpha_N$. So, summing up the above, we conclude the following:
\begin{theorem} \label{00001}
Given any solution of equation (\ref{0004}), a sequence $(y_n)$ defined by a substitution $y_n=t_nt_{n+3}/(t_{n+1}t_{n+2})$ is a solution of equation (\ref{675430}).
\end{theorem}
It goes without saying that in the case of a constant coefficient $\alpha_N$ we have the same relationship between the Somos-$N$ equation and equation (\ref{675430}).
From (\ref{07434}), we get
\begin{eqnarray}
H_N&=&\sum_{j=0}^{N-3}f_{j}f_{j+1}+\frac{\alpha_{N, 1}}{\prod_{j=1}^{N-3}f_{j}}   \nonumber\\
&=&\sum_{j=0}^{N-3} \frac{t_{j}t_{j+3}}{t_{j+1}t_{j+2}}+\alpha_{N, 1}\frac{t_{2}t_{N-2}}{t_{1}t_{N-1}} \nonumber\\
&=&\sum_{j=0}^{N-4} \frac{t_{j}t_{j+3}}{t_{j+1}t_{j+2}}+\alpha_{N, 0}\frac{t_{1}t_{N-3}}{t_{0}t_{N-2}}+\alpha_{N, 1}\frac{t_{2}t_{N-2}}{t_{1}t_{N-1}}+\beta_N \frac{t_{2}t_{N-3}}{t_{0}t_{N-1}}.
\label{764097}
\end{eqnarray}
Here, for simplicity of notation, we set $n=0$. One sees, that (\ref{764097}) generalizes expression (\ref{654009}) for non-autonomous case.
\begin{remark}
The proof of Theorem \ref{00001} is actually contained in the paper \cite{Hone3} except that equation (\ref{0986504}) is considered instead of equation (\ref{675430}). 
\end{remark}

\section{Volterra lattice and its solutions}

\label{77777}

The Volterra lattice  (\ref{60}), at one time was noted by physicists and used by them as a simplified model in some processes of plasma physics \cite{Manakov}, \cite{Zakharov}. In \cite{Manakov}, the method of the inverse scattering problem was applied to this equation. Due to this, the Volterra lattice can be considered as an integrable differential-difference equation. Later it was found that this equation is in the list of  equations sharing a property of having an infinite number of higher symmetries and local conservation laws. It was shown in \cite{Svinin1} that the Volterra lattice and many other integrable evolution equations and systems of equations can be obtained using an appropriate constraint on bi-infinite sequence of Kadomtsev-Petviashvili (KP) hierarchies.

It is clear that the solution of the Cauchy problem  for the Volterra lattice in the case of general initial conditions  is an unbearable task, but nevertheless, some special cases can be noted when the problem can be solved exactly, especially if the initial data itself is a solution of some  integrable equation. 
As was shown in \cite{Svinin1}, equation (\ref{675430}) arises as a result of further invariant restriction of the sequence of KP hierarchies and therefore is consistent   with  the Volterra lattice equation (\ref{60}). 
This means the following. Given some $N\geq 4$, let $Y_n(0)=y_{n}$ be an initial data satisfying recurrence (\ref{675430}) with  fixed constant $H_{N}$. Then a corresponding solution of the Cauchy problem $Y_n(x)$ must satisfy (\ref{675430}) for any $x\in\mathbb{R}$. Thus it makes sense to look for the solution of the Volterra chain in the form
\[
Y_n(x)=y_n+\sum_{j\geq 1}Y_{n, j}\frac{x^j}{j!},
\]
where $(y_n)$ is a solution of some of equations (\ref{675430}) hoping to sum it up later.

The following fact is useful for our aim. 
\begin{lemma}
Let  $(\tau_n(x))$ satisfies bi-linear differential-difference equation 
\begin{equation}
\tau_n\frac{\partial \tau_{n+1}}{\partial x}-\tau_{n+1}\frac{\partial \tau_{n}}{\partial x}=\tau_{n-1}\tau_{n+2}.
\label{67532}
\end{equation}
then  
\begin{equation}
Y_n=\frac{\tau_n \tau_{n+3}}{\tau_{n+1}\tau_{n+2}}
\label{60002}
\end{equation}
satisfies the Volterra lattice (\ref{60}). 
\end{lemma}
In other words, instead of looking for solution to the Volterra lattice, we  look for solution of bi-linear equation (\ref{67532}) of the form
\begin{equation}
\tau_n(x)=t_n+\sum_{j\geq 1}\tau_{n, j}\frac{x^j}{j!},
\label{678652}
\end{equation}
where $(t_n)$ is any sequence satisfying the non-autonomous Somos-$N$ equation (\ref{0004}). With a solution (\ref{678652}), due to what we got in  Section \ref{666666},  formula (\ref{60002}) gives  the corresponding solution of the Volterra lattice. 

Unfortunately we  do not have a clear algorithm for finding a solution, and so far we can only show the following result obtained by trial and error.
\begin{theorem} \label{1}
Let $(T_n)$ be any second-order linear sequence.
Let $B=B(x)$ be a  solution of the Riccati equation 
\begin{equation}
B^{\prime}=\frac{P}{Q}B\left(B-P\right)+P
\label{54}
\end{equation}
with the initial condition $B(0)=0$. Then 
\begin{equation}
\tau_{n}(x)=\left(T_n-T_{n-1}B(x)\right)e^{nx}.
\label{609849}
\end{equation}
is a solution of the differential-difference equation (\ref{67532}) with the condition $\tau_{n}(0)=T_n$. 
\end{theorem}
\begin{proof}
This theorem  can be proved by direct calculation. Substituting (\ref{609849}) into (\ref{67532}), we get the relation
\begin{eqnarray*}
&&\left(T_n-T_{n-1}B\right)\left((n+1)T_{n+1}-(n+1)T_{n}B-T_{n}B^{\prime}\right) \\[0.2cm]
&&\;\;\;\;\;\;\;\;\;\;\;-\left(T_{n+1}-T_{n}B\right)\left(nT_{n}-nT_{n-1}B-T_{n-1}B^{\prime}\right) \\[0.2cm]
&&\;\;\;\;\;\;\;\;\;\;\;\;\;\;\;\;-\left(T_{n-1}-T_{n-2}B\right)\left(T_{n+2}-T_{n+1}B\right)=0
\end{eqnarray*}
which needs to be proven. Since $B=B(x)$ is supposed to be a solution of the Riccati equation (\ref{54}), we can reduce  the last relation to the form
\begin{eqnarray}
&&-\left(\frac{P}{Q}B(B-P)+P\right)\left(T_n^2-T_{n-1}T_{n+1})+B^2(T_{n-1}T_n-T_{n-2}T_{n+1}\right) \nonumber \\[0.2cm]
&&\;\;\;\;\;\;\;\;\;\;\;\;\;-B\left(T_n^2-T_{n-2}T_{n+2}\right)+T_nT_{n+1}-T_{n-1}T_{n+2}=0.
\label{564211}
\end{eqnarray}
Using Vajda's identity (\ref{6400026}) for the second-order linear sequence, we can write
\[
T_n^2-T_{n-1}T_{n+1}=cQ^{n-1}D_1^2,\;\; T_{n-1}T_{n}-T_{n-2}T_{n+1}=cQ^{n-2}D_1D_2,
\]
\[
T_n^2-T_{n-2}T_{n+2}=cQ^{n-2}D_2^2,\;\; T_{n}T_{n+1}-T_{n-1}T_{n+2}=cQ^{n-1}D_1D_2,
\]
where $c=Qt_0^2-Pt_0t_1+t_1^2$. With these relations we bring (\ref{564211}) to the following form: 
\[
-\left(\frac{P}{Q}B(B-P)+P\right)Q^{n-1}D_1^2+B^2Q^{n-2}D_1D_2
-BQ^{n-2}D_2^2
+Q^{n-1}D_1D_2=0.
\]
Finally, substituting $D_1=1$ and $D_2=P$ into the last relation, we make sure that this is an identity. 
\end{proof}

\begin{corollary} \label{2}
A solution of the Cauchy problem for the Volterra lattice (\ref{60}) with initial data  $Y_n(0)=T_nT_{n+3}/(T_{n+1}T_{n+2})$ is given by
\begin{equation}
Y_n(x)=\frac{\left(T_n-T_{n-1}B(x)\right)\left(T_{n+3}-T_{n+2}B(x)\right)}{\left(T_{n+1}-T_{n}B(x)\right)\left(T_{n+2}-T_{n+1}B(x)\right)}.
\label{9988}
\end{equation}
\end{corollary}
Formula (\ref{9988}) gives an infinite set of solutions parametrized by  $(t_0, t_1, P, Q)\in \mathbb{R}^4$. We do not discuss here the physical meaning of this solution, if it exists at all. In any case, the physical solution must at least satisfy positivity condition $Y_n(x)>0$ in the domain of its definition, and we leave it as a separate task to identify such solutions. It is possible that solution of the form (\ref{9988}) may be of interest from the point of view of number theory.

\section{On Maclaurin series for $B(x)$}

\label{99999}

Let us discuss some technical details. Namely, let us discuss the Maclaurin series for the function $B(x)$ that is, by definition, a solution of the Riccati equation (\ref{54}). The first few terms of this series have the following form:
\begin{eqnarray}
B(x)&=&Px-\frac{P^3}{Q}\frac{x^2}{2!}+\frac{P^3\left(P^2+2Q\right)}{Q^2}\frac{x^3}{3!}-\frac{P^5\left(P^2+8Q\right)}{Q^3}\frac{x^4}{4!} \nonumber \\
&&+\frac{P^5\left(P^4+22P^2Q+16 Q^2\right)}{Q^4}\frac{x^5}{5!}-\cdots
\label{666432}
\end{eqnarray}
Defining new variables 
\[
A=-\frac{P}{Q}B,\;\; z= -\frac{P^2}{Q}x\;\; \mbox{and}\;\; q=\frac{Q}{P^2},
\]
from (\ref{666432}), we get
\[
A(z)=z+\frac{z^2}{2!}+(1+2q)\frac{z^3}{3!}+(1+8q)\frac{z^4}{4!}+(1+22q+16q^2)\frac{z^5}{5!}+\ldots 
\]
It is easy to check that this function, by  (\ref{54}), satisfies an equation
\[
\frac{dA}{dz}=1+A+qA^2. 
\]
From \cite{Sloane} http://oeis.org/A101280 it is known that this is the Riccati equation for the bi-variate generating function for a number triangle $(e_{n, j})$, where $n\geq 1$ and $0\leq j\leq \lfloor (n-1)/2 \rfloor$. To be more exacty, these numbers are defined by the recurrence relation
\[
e_{n, j}=(j+1)e_{n-1, k}+\left(2n-4j\right)e_{n-1, j-1}.
\]

The triangle of numbers $(e_{n, j})$ has the following  property: let $(E(n, j))$ be a triangle of the first-order Euler numbers \cite{Sloane} http://oeis.org/A173018 that can be defined, for example,  by  Worpitzky's identity \cite{Graham}
\[
x^n=\sum_{j=0}^{n-1}E(n, j){x+j\choose n}.
\]
They  can also be  defined by  the recurrence relation
\[
E(n, j)=(j+1)E(n-1, j)+\left(n-j\right)E(n-1, j-1).
\]
The relationship of these two number triangles is determined by the following formula:
\[
\sum_{j=0}^{n-1} E(n, j) x^j = \sum_{j=0}^{\lfloor (n-1)/2 \rfloor}e_{n, j}x^j(1+x)^{n-1-2j}.
\]

\section{On Maclaurin series for $\tau_n(x)$}

\label{10000000}

It is easy to calculate the first few terms of the Maclaurin series (\ref{678652}). We have, for example,
\[
\tau_{n, 1}=nT_n-PT_{n-1},\;\; \tau_{n, 2}=n^2T_n-\left(2Pn- \frac{P^3}{Q}\right)T_{n-1},\ldots
\]
By direct substitution, one can check the following. The sequences $(\tau_{n, 1})$ and $(\tau_{n, 2})$ satisfy linear recurrences
\[
\tau_{n+4, 1}=2P\tau_{n+3, 1}-\left(P^2+2Q\right)\tau_{n+2, 1}+2Q P\tau_{n+1, 1}-Q^2\tau_{n, 1}
\]
and
\begin{eqnarray*}
\tau_{n+6, 2}&=&3P\tau_{n+5, 2}-3\left(P^2+Q\right)\tau_{n+4, 2}+ P\left(P^2+6Q\right)\tau_{n+3, 2}-3\left(P^2+Q\right)Q\tau_{n+2, 2} \\
&&+3PQ^2\tau_{n+1, 2}-Q^3\tau_{n, 2},
\end{eqnarray*}
respectively. Inspired by this observations, we can assume the following.
\begin{conjecture}
Given any $r\geq 0$, let us define a set of polynomials $\{f_{r, j}(P, Q) : j=0,\ldots, 2r+2\}$ by
\[
F^{r+1}=\sum_{j=0}^{2r+2}f_{r, j}(P, Q) X^j,
\]
where $F=X^2-PX+Q$ is a characteristic polynomial of a linear sequence $(T_n)$. Then the sequence $(\tau_{n, r})$ satisfies $(2r+2)$-order linear recurrence
\[
\sum_{j=0}^{2r+2}f_{r, j}(P, Q) \tau_{n+j, r}=0.
\]
\end{conjecture}

\section*{Acknowledgements}

This work was supported by the Project No. 121041300058-1 of the Ministry of Education and Science of the Russian Federation.

\end{document}